\documentclass[11pt]{article}

\usepackage{amssymb}

\usepackage[usenames, dvipsnames]{color}

\usepackage{amsmath,latexsym,amsthm,amsfonts,mathtools}

\usepackage{graphicx,amsfonts}
\usepackage{hyperref}
\usepackage{enumitem, hyperref}

\numberwithin{equation}{section}

\newtheorem{remark}{Remark}[section]

\newtheorem{theorem}{Theorem}[section]

\newtheorem{corollary}[theorem]{Corollary}
\newtheorem{example}[theorem]{Example}

\def\be{\begin{equation}}
\def\ee{\end{equation}}

\usepackage[ulem=normalem,draft]{changes}

\newcommand{\SA}{\textcolor{black}}
\newcommand{\VM}{\textcolor{black}}

\title{On second order conditions for singular optimal control of port-Hamiltonian systems}

\author{M. Soledad Aronna \footnotemark[1]
\and Volker Mehrmann
\footnotemark[2]
}

\begin{document}

\date{}

\renewcommand{\thefootnote}{\fnsymbol{footnote}}
\footnotetext[2]{
Institut f\"ur Mathematik, MA 4-5, TU Berlin, Str. des 17. Juni 136,
D-10623 Berlin, FRG.
\texttt{mehrmann@math.tu-berlin.de}.
}

\footnotetext[1]{
Escola de Matemática Aplicada FGV EMAp, Praia de Botafogo, 190, 22250-900 Rio de Janeiro, Brazil. 
   \texttt{soledad.aronna@fgv.br}.}
\renewcommand{\thefootnote}{\arabic{footnote}}
\maketitle



\begin{abstract}                
We study nonlinear singular optimal  control problems of port-Hamil-tonian (descriptor) systems. We employ general control-affine cost functionals that include as a special case the energy supplied to the system. We first derive optimality conditions for the case of ordinary differential equations with and without control bounds by applying the general theory to the specially structured port-Hamiltonian case, and show that this leads to elegant optimality conditions, in particular in the linear case. We then extend these results to classes of nonlinear  port-Hamiltonian descriptor systems.
\end{abstract}

\section{Introduction}\label{sec:intro}
The energy-based modeling of physical systems employing the model class of port-Hamiltonian (descriptor) systems, see e.g. \cite{MehU23,SchaJelt2014}, has become an accepted paradigm  which has been successfully used in a  multitude of applications from a wide variety of practical domains: mechanics \cite{brugnoli2019por1t,brugnoli2019port2,macchelli2007port,siuka2011port},  electrical engineering \cite{reis2021analysis,schoberl2008modelling}, thermodynamics and fluid dynamics \cite{Altmann2017,hoang2011port,rashad2021port,rashad2021port2}, economics \cite{macchelli2014towards}. See \cite{duindam2009modeling,Jacob2012,SchaJelt2014} and the recent survey \cite{MehU23}. 

In this paper we consider  
a simplified form of  real  nonlinear port-Ha\-mil\-to\-ni\-an (pH) descriptor systems, see \cite{MehM19_CDC,MehU23}, that have a governing equation
\begin{equation}
\label{stateeq}
E\dot x = \big(J(x) - R(x) \big)Qx + G(x) u(t),
\end{equation}
with an associated energy function  
\begin{equation}\label{hamil}
\mathcal E: \mathbb{R}^n \to \mathbb R,\quad \ x\mapsto \frac12 x^TE^TQx,
\end{equation}
(often called \emph{Hamiltonian} or {\em storage function}). 
As is common in pH (descriptor) systems, the structure of the coefficients is the following:
\begin{itemize}
\item $J:\mathbb{R}^n \to \mathbb R^{n,n}$ is skew-symmetric;
\item $R:\mathbb{R}^n \to \mathbb R^{n,n}$ is symmetric and positive semidefinite;
\item $E,Q \in \mathbb R^{n,n}$ are such that $E^TQ$ is symmetric and positive semidefinite and represents the Hessian of 
$\mathcal E(x) =\frac 12 x^T E^TQx$, i.e., 
$\frac{\partial \mathcal E}{\partial x}(x)=E^TQ x$;
\item $G : \mathbb{R}^n \to \mathbb R^{n,m}$.
\end{itemize}
\SA{For \eqref{stateeq}-\eqref{hamil}, } the {\em output} $y$ is given by
\begin{equation}
\label{output}
y = G^T(x) Q x.
\end{equation}
Note that if $E$ is invertible, then one obtains the classical form of a pH system by a change of variables (see \cite{SchaJelt2014}).

The success of modeling within the class of pH (descriptor) systems is due to its many important properties which include the invariance of the class under power-conserving interconnection, which greatly simplifies modularized automated modeling, and the invariance under Galerkin projection which makes the class very suitable for discretization and model reduction, see \SA{\cite{CarHLMR24,RasCSS20}}. Physical properties like \emph{energy dissipation, Lyapunov stability and passivity} are encoded in the  structure of the equations leading directly to  another key property of port-Hamiltonian (descriptor) systems, the \emph{power balance equation} and the resulting \emph{dissipation inequality}, 
see e.g. \cite{MehM19_CDC,MehU23} or \cite{SchaJelt2014} for the  case of ordinary pH systems.
%
  %
  %
  %
From a   physics point of view, pH systems  model the interaction of three types of energy/power. The \emph{stored energy} (in the energy storing components) is represented by the nonnegative quadratic form $\mathcal E(x)$, the \emph{dissipated energy}  by the nonnegative quadratic form  
$ \mathcal D(x)=-       x^T Q^T R Q x$
and the \emph{supplied energy} by the form
$\mathcal S(y,u)= y^T u$.

\begin{example}\label{ex:mbs}{\rm 
A classical example of pH modeling  are (dissipative) Hamiltonian equations of motion, see e.g. \cite{AbrM08,SchaJelt2014}, which, in first order descriptor representation u\-sing the position coordinate $q$ and velocity coordinate $\tilde p$ 
(replacing the usual linear momentum $p=M\dot q$),
together with a force term $\tilde B(q) u$, take the
descriptor form
\begin{equation}\label{daemechsys}
\begin{bmatrix} 
M & 0 \\ 0 &I\end{bmatrix}
\begin{bmatrix}
    \dot {\tilde p} \\ \dot {q}
\end{bmatrix}=
\begin{bmatrix}
-D(p,q)  & -I \\ I & 0
\end{bmatrix}
\begin{bmatrix}
I & 0 \\ 0 & K
\end{bmatrix}
\begin{bmatrix}
    \tilde p \\ q
\end{bmatrix}
+ \begin{bmatrix}  \tilde B(q)\\ 0 
\end{bmatrix}u
\end{equation}
coupled with a collocated output 
\begin{equation}\label{daeoutmechsys}
y=
\begin{bmatrix}
\tilde B(q)^T & 0 
\end{bmatrix}\begin{bmatrix}
I & 0 \\ 0 & K
\end{bmatrix}\begin{bmatrix}
    \tilde p \\ q
\end{bmatrix}.
\end{equation}
%
%
%
Here the energy function (Hamiltonian) is 
\[
\mathcal E(q,\tilde p) =\frac 12 {\tilde p}^TM \tilde p+ \frac 12 q^T K q,
\]
i.e., kinetic plus potential energy, where $M$ is a positive (semi-)definite mass matrix, $K$ is a positive semidefinite stiffness matrix, and  $D(q,p)=D(q,p)^T$ is a positive semidefinite coefficient matrix that models internal damping.
This representation has the advantage that small masses can be set to $0$, leading to a singular mass matrix $M$. See \cite{MehS23} for a detailed analysis of  different representations and transformations between the different formulations.
A typical optimal control application is arising in the control of the system into an equilibrium state, e.g. in the stabilization of buildings excited by earthquakes, see e.g. \cite{MahZS00,SchZBSTW24}.}
\end{example}

With an increased understanding of the many advantageous properties of modeling with pH (descriptor) systems, in recent years also feedback control and optimal control problems for this model class have become an important research topic. 
\VM{While the theoretical analysis from different perspectives is studied in \cite{FauKMPSW25,FauMPSW22,wu2020reduced}, and we will come back to this later, 
the use in  different application fields is studied, e.g.,  in dynamic network flow problems \cite{DogKLST23}, learning control \cite{KoeSSH21,NagLJB15}, electrothermal microgrids \cite{KriS21}, and thermodynamics \cite{Masc22}.}

In the  
 optimal control of  pH (descriptor) systems,  a natural \emph{cost function} is the integral of the supplied energy 
 {\small
\[
\mathcal{J}(x, u) = 
\int_{t_0}^{T} \mathcal{S}(y(\tau), u(\tau)) \, d\tau 
= \int_{t_0}^{T} x(\tau)^T Q^T G(x) u(\tau) \, d\tau.
\]
}
It is readily available as a mathematical expression in terms of the ports  and it is interesting from an application point of view to employ it as the objective to be minimized in optimal control problems (OCPs) \cite{FauKMPSW25,FauMPSW22,lhmnc}, \SA{where as cost function one typically uses 
\[
\tilde J(x,u) := \frac 12
\begin{bmatrix}
   x^T &u^T 
\end{bmatrix}\begin{bmatrix}\tilde Q & \tilde S \\
\tilde S^T &\tilde R \end{bmatrix} \begin{bmatrix}
    x\\ u
\end{bmatrix}.
\]
with semidefinite $\begin{bmatrix}\tilde Q & \tilde S \\
\tilde S^T &\tilde R \end{bmatrix}$.
While it is in general no problem to drop the quadratic term in $x$, i.e. by setting $\tilde Q=0$ \cite{Meh91}, it is usually not possible to set $\tilde R=0$, as is done when the supplied energy is used as cost function. In this case  the resulting OCP is typically singular,} since it is missing the (commonly used) positive quadratic  regularization term in $u$, and standard solution techniques, for instance the construction of a Riccati state feedback, see e.g. \cite{KunM11a,KunM24,Meh91}, are not directly applicable.  
In view of this difficulty, the topic of this paper is the analysis of singular  OCPs with constraints given by pH (descriptor) systems without the regularization approach, \SA{i.e. without the term $u^T\tilde Ru$ in the cost function. }

The paper is organized as follows. 
In Section~\ref{SecGohLC} we recall second order optimality conditions \SA{for general control-affine problems.
In Section~\ref{sec:optconph},} we analyze how these optimality conditions look when we specialize the constraint function from general control-affine ordinary differential equations to ordinary port-Hamiltonian  systems. We show that the structure helps to obtain much more elegant conditions. We extend these results to port-Hamiltonian descriptor systems in Section~\ref{sec:deslin} and we present some conclusions and directions of future work in Section~\ref{sec:conclusion}.

\SA{\section{On Goh and generalized Legendre-Clebsch conditions}
}
\label{SecGohLC}

We obtain our results by applying optimality conditions for singular optimal control 
of control-affine systems\SA{, known as {\it Goh} and {\it generalized Legendre-Clebsch conditions, that we briefly introduce in the sequel. We consider}} problems of the form
\begin{align}
    \min  & \quad \Psi(x(T)) \label{finalcost} \\
    \mathrm{s.t.}  & \quad \dot x = f(x,u) = f_0(x) + \sum_{i=1}^m  f_i(x) u_i, \label{generalcontrolaffine} \\
                   & \quad x(0) = x_0. \label{initialcondition}
\end{align}
with $f_i: \mathbb{R}^n \to  \mathbb R^n$, for $i=0\ldots,m$ and $\Psi : \mathbb{R}^n \to \mathbb R.$
\SA{
The {\em state} $x$ belongs to the state space $\mathcal{X} := W^{1,\infty}([0,T];\mathbb{R}^n)$ (the Sobolev space of functions in $L^\infty([0,T];\mathbb{R}^m),$ with weak first order derivatives that have a finite $L^\infty$-norm, with $L^\infty$ denoting the usual Lebesgue space) and the {\em input (control)} $u = (u_1,\dots,u_m)^T$ ranges in the control space $\mathcal{U} := L^\infty([0,T];\mathbb{R}^m)$.}
At this point we recall that any problem with a general control-affine integral cost can be rewritten in the Mayer form \eqref{finalcost}-\eqref{initialcondition} by adding an extra scalar state variable (see \eqref{nonlinearpHproblem} below). Therefore, the formulation \eqref{finalcost}-\eqref{initialcondition}
covers problems with integral-type costs of the form
$
\mathcal J\coloneqq \int_0^T \ell_0(x) + \sum_{i=1}^m \ell_i (x) u_i\, dt,
$
where $\ell_i: \mathbb{R}^n \to \mathbb{R}$ for $i=0\ldots,m$.

\SA{
For nonlinear systems as the one in \eqref{generalcontrolaffine}, one cannot obtain the control straightforwardly from Pontryagin's Maximum Principle (PMP).  More precisely, the maximum condition provided by the PMP results in a maximization with respect to the control variable of the {\it Hamilton} function, which is itself a control-affine function
\begin{equation}
\label{Haffine}
\mathcal{H} = 
\lambda^T f = \lambda^T ( f_0(x) + \sum_{i=1}^m  f_i(x) u_i ),
\end{equation}
$\lambda$ here being the {\it adjoint state,} whose associated {\em adjoint equation} is given by
\begin{equation}
\label{adjointeq}
    -\dot{\lambda}^T = \partial_x \mathcal{H} = \lambda^T \partial_x f,\quad \lambda(T) = \nabla \Psi(x(T)).
\end{equation}
Consequently, whenever 
 \begin{equation}\label{Hu0}
     \mathcal{H}_{u_i} = 
     \lambda^T f_i = 0,
 \end{equation}
holds along a non-trivial interval of time for some $i\in \{1,\dots,m\},$ no information about the optimal control can be retrieved from \eqref{Haffine}. When this situation occurs along an optimal trajectory, it is said that the corresponding control has a {\it singular arc.} Briefly speaking, singular arcs of a control component $u_i$ correspond to intervals in which $\frac{\partial \mathcal{H}}{\partial u_i}$ vanishes and the corresponding Hessian matrix $\frac{\partial^2 \mathcal{H}}{\partial u^2}$ is singular (which holds trivially in the current control-affine framework).
In the absence of control constraints, \eqref{Hu0} results in a necessary condition for optimality.
Since \eqref{Hu0} does not depend explicitly on the control variable, one can differentiate with respect to time and obtains
  \begin{equation}
  \label{dotHu0}
0= \frac{d}{dt} \mathcal{H}_{u_i} 
= \dot \lambda^T f_i + \lambda^\top f_i' \dot x = \lambda^T \big( [f_0,f_i] + \sum_{j=1}^m u_j  [ f_j, f_i] \big),
 \end{equation}
where we used the state equation \eqref{generalcontrolaffine} and the adjoint equation \eqref{adjointeq} to replace $\dot x$ and $\dot \lambda,$ and where $f_i'$ denotes the Jacobian, i.e., the derivative of $ f_i$ with respect to $x$, and for any pair of smooth vector fields $h,k \colon \mathbb{R}^{n} \to \mathbb{R}^{n}$,  $[h,k]:=k'h-h'k$ is their {\em Lie bracket}, with $k',h'$ denoting the Jacobians of $k,h$.
A remarkable property, discovered by Goh \cite{Goh66} (and further explored by Robbins \cite{Rob67}), is that the coefficients of the control in the first time derivative \eqref{dotHu0} of $\mathcal{H}_{u_i}$ vanish. More precisely, {\em Goh conditions} state the following, see e.g. \cite{ABDL12,FraTon13,schattler2012geometric} for rigorous proofs.
\begin{theorem}[Goh condition]
\label{ThGoh}
For the optimal control problem  \eqref{finalcost}-\eqref{initialcondition}, the
\emph{Goh condition}
\begin{equation}
\label{Gohgeneral}
\lambda^T[ f_i, f_j] = 0, \quad \text{for } i,j =1,\dots,m
\end{equation}
holds along optimal trajectories.
\end{theorem}
}

\SA{\begin{remark}
    \SA{At this point it is worth recalling that the traditional technique of adding a regularizing term of the form $\alpha \sum_{i=1}^m u_i^2$ with positive $\alpha$ and then taking $\alpha \to 0$ does not work in this context, since \eqref{Hu0} will depend on $u$ and then the expression \eqref{dotHu0} would not be possible. On the other hand, if one aims at obtaining a feedback formula for the controls on singular arcs by applying Pontryagin's Maximum Principle to the modified quadratic cost and taking $\alpha \to 0,$ the expression contains $\alpha$ in the denominator and the limit may fail to exist, see, e.g. \cite{JacGL70}. 
Despite this difficulty, it has been shown in the case of linear constant coefficients that this problem has some hidden regularity in a subspace \cite{FauKMPSW25}.}
\end{remark}}

\SA{The Goh property \eqref{Gohgeneral} guarantees then that \eqref{dotHu0}  does not explicitly contain the control variables, which allows one to differentiate the expression \eqref{dotHu0} once more. In the second order derivative $\frac{d^2}{dt^2} \mathcal{H}_{u},$ the matrix coefficient of the control variable turns out to be semidefinite.
Indeed, one gets
\begin{equation}
    \label{ddotHu}
\frac{d^2}{dt^2} \mathcal{H}_{u_i} =
\lambda^T [f_0,[f_0,f_i]] + \sum_{j=1}^m u_j \lambda^T [f_j,[f_0,f_i]].
\end{equation}
Thus, from $\frac{d^2}{dt^2} \mathcal{H}_{u_i} = 0$ one obtains a linear system for $u$ given by \begin{equation}
\label{ufeedback}
Wu + d = 0,
\end{equation}
with the \emph{Hessian matrix} $W=[w_{ij}]$ and the right hand side vector
$d$, where for  $i,j=1,\dots,m$,
\begin{equation}
\label{Wd}
w_{ij} \coloneqq \lambda^T [f_j,[f_0,f_i]],\qquad 
d_i \coloneqq \lambda^T [f_0,[f_0,f_i]].    
\end{equation}
This leads to a \emph{generalized Legendre-Clebsch condition}, see
e.g. \cite[Theorem 6.2 and/or Corollary 6.3]{Krener1977}, \cite{ABDL12}, or  \cite{FraTon13}.
\begin{theorem}[Generalized Legendre-Clebsch condition]
\label{thm:genLC}
For the nonlinear control-affine problem \eqref{finalcost}-\eqref{initialcondition}, the associated Hessian matrix $W$ in \eqref{Wd} is negative semidefinite
along optimal trajectories.
Additionally, if $W$ is actually negative definite,
then, from \eqref{ufeedback}, one can express the optimal control in \emph{feedback form}.
\end{theorem}
}

\SA{
\begin{remark}
    Note that, whenever $W$ is actually negative definite, the feedback obtained from \eqref{ufeedback} will depend on the state and adjoint variables.
\end{remark}
}

\SA{The extensions of Theorems \ref{ThGoh} and \ref{thm:genLC} to control-constrained problems are provided in the subsequent sections, in which we use Goh and the generalized Legendre-Clebsch conditions to deduce properties for pH systems.}


\section{Optimality conditions for nonlinear  pH  systems}\label{sec:optconph}
The optimality conditions for singular optimal control problems  \SA{in the last section} do not exploit the pH structure.  In this section, we therefore make explicit use of the nonlinear pH structure in~\eqref{stateeq} and generalize results of \cite{FauKMPSW25,Masc22,SchPFWM23}.
 
\subsection{Optimality conditions for ordinary nonlinear  pH systems}

We first consider the following optimal control problem governed by the nonlinear ordinary pH dynamics \eqref{stateeq} with $E=I$, 
minimizing a general control-affine integral cost
\begin{equation}
\label{nonlinearpHproblem}
\begin{aligned}
        \min\quad  & x_{n+1}(T)  \\
 \mathrm{s.t.} \quad  & \dot x = \big(J(x) - R(x) \big) \frac{\partial  \mathcal {E}}{\partial x} (x) + G(x) u(t) ,\\
 & \dot x_{n+1} = \ell_0(x) + \sum_{i=1}^m \ell_i (x) u_i, \\
 &  
 \begin{bmatrix}
     x \\ x_{n+1}
 \end{bmatrix}(t_0) = 
 \begin{bmatrix}
     x_0 \\ 0
 \end{bmatrix},
\end{aligned}
\end{equation}
with the assumption that the control $u$ takes values in an {\em open set} $\mathbb  U\subseteq \mathbb R^{m}$.
For the sake of simplicity in the presentation of the results, we introduce the abbreviations 
\begin{equation}
g_0 (x) \coloneqq   \big(J(x) - R(x) \big) Qx =  \big(J(x) - R(x) \big) \frac{\partial {\mathcal E}}{\partial x}(x),
\end{equation}
and,  setting
$z \coloneqq \begin{bmatrix}
     x \\ x_{n+1}
 \end{bmatrix}$
and writing $G = \begin{bmatrix}
    g_1\, \dots \, g_m
\end{bmatrix}$, 
we define
\begin{equation}
\label{deffi}
f_0(z)  \coloneqq \begin{bmatrix} g_0 (x)\\ \ell_0(x) \end{bmatrix};\qquad f_i (z) \coloneqq  \begin{bmatrix} g_i (x)\\ \ell_i(x) \end{bmatrix}.
\end{equation}
In this case, recalling the {\em Hamilton function} introduced in \eqref{Haffine}, the adjoint variable has an additional coordinate $\lambda_{n+1}$ and takes the form
$
\begin{bmatrix} \lambda \\ \lambda_{n+1} \end{bmatrix},
$
where $\lambda$ maps into $\mathbb{R}^n$ and $\lambda_{n+1}$ is scalar.
In the absence of (state or terminal) constraints, the multiplier $\lambda_{n+1}$ is unique and the associated extremal  is {\em normal,} i.e., the corresponding multiplier $\lambda_{n+1}$ can be set to $1$. 
Therefore, throughout the remainder of the article, we fix $\lambda_{n+1}=1$.

\SA{
\begin{remark}\label{rem:supplied}
{\rm Note that the use of the supplied energy as cost function in \eqref{nonlinearpHproblem} corresponds to the special case in which $\ell_0(x)=0$ and $\ell_i(x)=y_i(x)$, for $i=1,\dots,m$. We will come back to this choice in more detail in the following section.}
\end{remark}
}
\begin{corollary}[Goh conditions for nonlinear ordinary pH control]
\label{ThmGohpH}
    For the optimal control problem with  nonlinear pH constraint \eqref{nonlinearpHproblem}, it holds that
\begin{equation}\label{GohNLpH}
\left\{
\begin{array}{rl}
    \lambda^T [g_i,g_j] &=0,\\
        \ell_j' \, g_i - \ell_i' \, g_j & =0,
\end{array}
\right.
\qquad \text{for } \, i,j =1,\ldots, m,
\end{equation}
along optimal trajectories.
\end{corollary}
\begin{proof}
The result follows from  Theorem \ref{ThGoh}.
\end{proof}
Using Corollary~\ref{GohNLpH}, we compute the expressions for the {\em generalized Legendre-Clebsch conditions}
(see Theorem \ref{thm:genLC}, and $W$ and $d$ given in \eqref{Wd}) in the context of problem \eqref{finalcost}--\eqref{initialcondition}. We have
\begin{equation}
\label{nonpHbracket}
        [f_0,f_i] = 
        \begin{bmatrix}
            [g_0,g_i] \\
            \ell_i' g_0 - \ell_0' g_i
        \end{bmatrix},
\end{equation}
and 
\[
        [f_j,[f_0,f_i]] =  
        \begin{bmatrix}
            [g_j,[g_0,g_i]] \\
            \ell_i'' g_0 g_j + \ell_i' g_0' g_j - \ell_0'' g_i g_j - \ell_0' g_i' g_j - \ell_j' [g_0,g_j] 
        \end{bmatrix}.
\]
We set
\begin{multline}
\label{WpH}
     w_{ij} \coloneqq
    \lambda^T [g_j,[g_0,g_i]]
      + \Big( \ell_i'' g_0 g_j + \ell_i' g_0' g_j - \ell_0''g_i g_j - \ell_0'g_i'g_j - \ell_j'[g_0,g_i] \Big),   
\end{multline} 
      \begin{multline}
          d_{i} \coloneqq \lambda^T [g_0,[g_0,g_i]] + \Big( \ell_i'' g_0 g_0 + \ell_i' g_0' g_0 - \ell_0''g_i g_0 - \ell_0'g_i'g_0 - \ell_0'[g_0,g_i] \Big),
      \end{multline}
and obtain the following result.
\begin{corollary}[\SA{Generalized} Legendre-Clebsch condition for nonlinear ordinary pH control]\label{ThmLCpH}
    For the nonlinear pH optimal control problem \eqref{nonlinearpHproblem}, 
    the matrix $W=[w_{ij}],$ with $w_{ij}$ given in \eqref{WpH}, is negative semidefinite. 
    Additionally, whenever $W$ is negative definite over some control interval then, on that interval,
     one can write the optimal control in \emph{feedback form} from \eqref{ufeedback}.
\end{corollary}

\begin{proof}
The proof follows from Theorem \ref{thm:genLC}.
\end{proof}

\begin{remark}\label{rem:feedback}
 Note that when \( W \) is non-singular, the control \( u \) can be expressed in feedback form as a function of the state and costate variables. This representation is useful for establishing the regularity of the control and facilitates the application of numerical techniques, such as the shooting method.
\end{remark}

\SA{While the Goh and generalized Legendre-Clebsch conditions from Corollaries \ref{ThmGohpH} and \ref{ThmLCpH} do not offer much insight for general nonlinear pH systems, they simplify significantly, or even hold trivially, in specific examples where symmetry properties naturally arise. This simplification reveals an interesting novelty, which is exemplified in the general nonlinear model for dissipative Hamiltonian equations of motion, introduced above in Example \ref{ex:mbs}.}
\begin{example}\label{ex:lie}{\rm
    In Example \ref{ex:mbs}, with the associated cost of minimizing the energy supply 
    \[\int_0^T y^\top u \,dt = \int_0^T \frac{\partial \mathcal{E}}{\partial p}^\top B(q) u \,dt=\int_0^Tp^\top M^{-1}B(q)u\,dt
    \]
    the Lie brackets $[f_i,f_j]$ are given by
\begin{equation*}
        [f_i,f_j] 
        = \begin{bmatrix}
            0\\
            0\\
            B_j(q)^\top 
\frac{\partial^2 \mathcal{E}}{\partial p^2}
B_i(q) -B_i(q)^\top  \frac{\partial^2 \mathcal{E}}{\partial p^2} B_j(q)
\end{bmatrix}.
\end{equation*}
These vanish identically for any sufficiently regular $\mathcal E$. In the current framework, one has \SA{$\frac{\partial^2\mathcal{E}}{\partial p^2} = M^{-1},$} for $M$ being symmetric positive definite, so {\em Goh conditions \eqref{GohNLpH} hold trivially}. 
    One gets the following expression for the second-order Lie brackets
    \begin{multline*}
        [f_j,[f_0,f_i]] =
    \begin{bmatrix}
        B_i' M^{-1} B_j + B_j'M^{-1} B_i  \\
        0
    \end{bmatrix} \\
    +
    \begin{bmatrix}
         \frac{\partial D}{\partial p}(M^{-1} B_i) B_j + \frac{\partial^2 D}{\partial p^2}(M^{-1} p, B_i) B_j \\
        0
    \end{bmatrix},
    \end{multline*}
    which considerably simplifies under additional assumptions on $D$ or $B$, and may lead to very simple expressions for the singular controls.}
\end{example}

\if{
We have
\begin{equation}
[f_0,f_i] = 
\begin{pmatrix}
\left[ J\frac{\partial H}{\partial x}  ,G_i \right] \\
\nabla \ell_i  J \frac{\partial H}{\partial x} - \nabla \ell_0 G_i
\end{pmatrix}, 
\end{equation}
and we get the following expressions for $d$ and $W$:
\begin{multline}
\label{dnonlinearpH}
d_i =p^T \big[f_0,[f_0,f_i] \big] = \\
\begin{pmatrix}
\left[ J\frac{\partial H}{\partial x}  ,\left[ J\frac{\partial H}{\partial x}  ,G_i \right] \right] \\
 \frac{\partial^2 \ell_i}{\partial x^2}   (J \frac{\partial H}{\partial x})^2 
+\nabla\ell_i \left( J' \frac{\partial H}{\partial x}  + J \frac{\partial^2 H}{\partial x^2} \right) J \frac{\partial H}{\partial x} 
- \frac{\partial^2 \ell_0}{\partial x^2} G_i  J\frac{\partial H}{\partial x} 
  + \nabla\ell_0 \left( J' \frac{\partial H}{\partial x}  + J \frac{\partial^2 H}{\partial x^2} \right)G_i
\end{pmatrix}, 
\end{multline}
\begin{multline}
\label{WnonlinearpH}
W_{ij} = p^T \big[f_j,[f_0,f_i] \big] = \\
p^T
\begin{pmatrix}
 \left[ G_j  ,\left[ J\frac{\partial H}{\partial x}  ,G_i \right] \right] \\
 \left\{ \frac{\partial^2 \ell_i}{\partial x^2}   J \frac{\partial H}{\partial x} 
+\nabla\ell_i \left( J' \frac{\partial H}{\partial x}  + J \frac{\partial^2 H}{\partial x^2} \right) 
- \frac{\partial^2 \ell_0}{\partial x^2} G_i 
  - \nabla\ell_0 G_i' \right\} G_j - \nabla \ell_j [J \frac{\partial H}{\partial x} , G_i]
\end{pmatrix}.
\end{multline}
}\fi

\subsection{Optimality conditions for nonlinear ordinary pH systems with control bounds}

In this section we discuss problem \eqref{nonlinearpHproblem} in the presence of bounds on the control. This means that we no longer assume that $\mathbb{U}$ is open, but  consider a  set of admissible control values 
\begin{equation}
\label{U}
\mathbb{U} = [\underline{u}_1,\overline{u}_1] \times \dots [\underline{u}_m,\overline{u}_m], 
\end{equation}
with real numbers $\underline u_i \leq \overline u_i,$ for each $i=1,\dots,m$. 
\SA{In this case, the Goh optimality condition is restricted to the components that lie in the interior of the control set and one has the following optimality condition. }
%

\begin{theorem}[Goh condition for ordinary pH control under control constraints]
\label{PropGohnonlinearpHconstraints}
Consider the nonlinear pH optimal control problem \eqref{nonlinearpHproblem}  under the control constraint $u(t) \in \mathbb{U}$, for a.e. $t\in [0,T],$ with the set $\mathbb{U}$ as in \eqref{U}. 
Then, on any open interval over which an optimal control $u^*$ verifies
\[
\underline{u}_i < u^*_i(t) < \overline{u}_i \quad \text{ and } \quad \underline{u}_j < u^*_j(t) < \overline{u}_j,
\]
for some pair $i,j=1,\ldots,m,$ \SA{the Goh conditions \eqref{GohNLpH}} must necessarily hold.
\end{theorem}
\SA{
The proof of Theorem~\ref{PropGohnonlinearpHconstraints} is not a straightforward consequence of the unconstrained version of Theorem \ref{ThGoh}. For more details, the reader is referred to \cite{ABDL12,FraTon13}.
}

\SA{One gets the following generalized Legendre-Clebsch condition for the square submatrix of $W$ corresponding to singular components.}

\begin{theorem}[Generalized Legendre-Clebsch condition and feedback for\-mu\-la for nonlinear ordinary pH control under control constraints]
\label{GLCnlconstraints}
\SA{
Under the hypotheses of Theorem~\ref{PropGohnonlinearpHconstraints}, let $(a,b)$ be an open interval on which
\begin{equation*}
   \left\{
   \begin{array}{ll}
   \underline{u}_i < u^*_i(t) < \overline{u}_i, \quad  &\text{for } i\in S,\\
    u^*_i(t) = \underline{u}_i, \quad  &\text{for } i\in \underline{B},\\
    u^*_i(t) = \overline{u}_i, \quad  &\text{for  } i\in \overline{B}.
    \end{array}
    \right.
\end{equation*}
Thus, $\{S,\underline{B},\overline{B}\}$ is a partition of $\{1,\dots,m\}$ into {\em singular, lower-bang} and {\em upper-bang} components of the control $u^*.$}
One gets, \SA{for $i\in S$,} that
\begin{eqnarray}
&&         \lambda^T [f_0,[f_0,f_i]]  + \sum_{j \in S } u_j^*  \lambda^T [ f_j,[ f_0, f_i]]  \label{expressionusing} \\
   &&   + \sum_{j \in \underline{B} } \underline{u}_j   \lambda^T [ f_j,[ f_0, f_i]]  + \sum_{j \in \overline{B} } \overline{u}_j  \lambda^T [f_j,[f_0,f_i]] = 0. \nonumber
\end{eqnarray} 
holds on $(a,b)$.
 Additionally, the $|S| \times |S|$-matrix $W_S(t)$ with entries $w_{ij}=\lambda^T [ f_j,[ f_0, f_i]],$ for $i,j \in S$, is negative semidefinite.
\end{theorem}
\begin{proof}
The expression \eqref{expressionusing} follows from calculations analogous to \eqref{Hu0}--\eqref{dotHu0}, \eqref{ddotHu} and Goh condition of Theorem \ref{GohNLpH}. The negative semidefiniteness of $W_S$ holds in view of \cite[Theorem 4.1]{FraTon13}.
\end{proof}

\begin{remark}[A feedback formula for the singular controls]\label{rem:fbformula}{\rm 
    The matrix $W_S$ in Theorem \ref{GLCnlconstraints} is the submatrix of $W$ in \eqref{Wd} corresponding to the singular components of the control. Whenever  $W_S$ is negative definite, by using expression \eqref{expressionusing}, one can write the singular control in terms of the state variable $x$, the adjoint variable $
\lambda$ and the bang controls.
    }
\end{remark}

\subsection{Optimality conditions for nonlinear pH descriptor systems}\label{sec:dae}
In this section we extend the results from the previous two subsections to a special class of pH descriptor systems. For general descriptor systems, optimality conditions have been derived in 
\cite{KunMeh08}, see  \cite{KunM24} for a detailed exposition. But again these do not exploit the particular structure.

It has been shown in \cite{CamKM12} that by an appropriate regularization procedure, using transformations, derivatives, and feedback designs, a general descriptor system can always be 
transformed to a so-called \emph{index-one} or \emph{strangeness-free} system, see \cite{KunM24}.
Then, for a system of the form \eqref{stateeq}--\eqref{output}, after appropriate coordinate changes,  we may assume 
 that the system has the structure
\begin{eqnarray}
&& \begin{bmatrix} I_{n_1} & 0 \\ 0 & 0 \end{bmatrix}\begin{bmatrix}\dot x_1\\ \dot x_2\end{bmatrix} = \begin{bmatrix} G_{1}\\ G_{2}\end{bmatrix} (x) \, u+\label{stateeqnew}\\
&&\left (
    \begin{bmatrix} J_{11} & J_{12} \\ -J_{12}^T & J_{22} \end{bmatrix}(x)-
    \begin{bmatrix} R_{11} & R_{12} \\ R_{12}^T & R_{22} \end{bmatrix}(x)\right )
    \begin{bmatrix} Q_{11} & Q_{12}\\ Q_{21}&Q_{22} \end{bmatrix}\begin{bmatrix} x_1\\ x_2\end{bmatrix}\nonumber 
\end{eqnarray}
with  {\em output} 
\begin{equation}
\label{outputnew }
y = \begin{bmatrix} G_{1}(x)^\top & G_{2}(x)^\top\end{bmatrix} \begin{bmatrix} Q_{11} & Q_{12}\\ Q_{21}&Q_{22} \end{bmatrix}\begin{bmatrix} x_1\\ x_2\end{bmatrix}.
\end{equation}

Observe that the condition $Q^TE=E^TQ\geq 0$ implies that $Q_{12}=Q_{21}^T=0$ and $Q_{11}\geq 0$. In the following, to keep the presentation  simple, we also assume that $G_2(x)=0$. If this is not the case, then the following reduction procedure will lead to a system with a feed-through term for which an extra transformation is necessary, see \cite{MehU23}. The assumption that the system is of index at most one  means that from the algebraic equation 
\begin{align*}
0=&[(J_{22}(x_1,x_2)-R_{22}(x_1,x_2))Q_{21}+(-J_{12}^T(x_1,x_2)-R_{12}^T(x_1,x_2))Q_{11}x_1\\
&+[(J_{22}(x_1,x_2)-R_{22}(x_1,x_2))Q_{22}x_2,
\end{align*}
the variable $x_2$ can be expressed 
via the Implicit Function Theorem as a function of $x_1$ and inserted in the coefficients. Due to the structure of $G(x)$,
$x_2$ contributes to the cost functional only indirectly as argument of $G$.  Thus, we can first solve the optimal control problem for $x_1$ and 
obtain $x_2$ in a post-processing step.
This means that the dynamics  is described by the ordinary pH system
\begin{equation*}
\begin{aligned}
    \dot{x}_1 &= (\tilde{J}_{11}(x_1) - \tilde{R}_{11}(x_1)) Q_{11} x_1 + \tilde{G}_1(x_1) u, \\
    y &= \tilde{G}_1^T(x_1) Q_{11} x_1,
\end{aligned}
\end{equation*}
and the cost functional is given by
\[
\mathcal J(x_1,u)=\int_{t_0}^{T} x_1^T(\tau)  Q_{11}^T \tilde G(x_1(\tau))u(\tau) d\tau.
\]
Thus we have reduced the problem
to one of the form \eqref{nonlinearpHproblem} for which the previous simplified formulas directly apply.

\section{Optimality conditions for linear pH systems}\label{sec:deslin}

In the last section we have presented the Goh conditions for pH (descriptor) systems which turned out to be much simpler than those for general systems.  This simplicity is even more pronounced in the case of linear pH systems.
 We again start with the optimal control of ordinary linear pH systems.
\subsection{Optimality conditions for linear pH systems without control constraints}
 \label{sec:linph}
 In this subsection we analyze the Goh optimality conditions for problems when the constraint equation is given by a linear time-invariant pH system and the energy function $\mathcal E=\frac 12 x^TQx$ is again quadratic in the state, i.e.,
 \begin{align}
\dot{x}(t) &= (J - R) Q x(t) + G u(t), \quad x(0) = x^0, \label{linearpH} \\
y(t) &= G^T Q x(t), \nonumber
\end{align}
 where
 \begin{itemize}
 \item $J \in \mathbb R^{n,n}  $ is skew-symmetric;
 \item $ R \in  \mathbb R^{n,n} $ is symmetric positive semidefinite;
 \item $ Q \in \mathbb R^{n,n} $ is symmetric positive definite;
 \item $ G =\begin{bmatrix}g_1& \ldots& g_m\end{bmatrix}  \in \mathbb R^{n,m}$ has full rank $m \leq n;$
 \item $ u(t) \in \mathbb{U}$ for a.e. $t,$ where $\mathbb{U}$ is an open subset of $\mathbb{R}^m.$ 
 \end{itemize}
Note that more general linear port-Hamiltonian formulations are possible, 
we follow here the formulation considered in \cite{FauKMPSW25,FauMPSW22,SchPFWM23}.

We study the optimal control problem with quadratic cost
\begin{equation} 
\label{linearpHOCP}
\begin{aligned}
\min \quad & \int_0^T \left( y^T Y y + y^T N u + \ell^T u \right) \, dt, \\
& \text{s.t.} \quad \eqref{linearpH},
\end{aligned}
\end{equation}
where the matrices involved in the cost functional satisfy
 \begin{itemize}
 \item $Y \in \mathbb{R}^{n,n}$  is symmetric and positive semidefinite;
 \item $N=\begin{bmatrix}n_1& \ldots& n_m\end{bmatrix} \in \mathbb R^{m,m}$; 
 \item $ \ell=\begin{bmatrix}\ell_1& \ldots& \ell_m\end{bmatrix}^T   \in \mathbb{R}^{m}$. 
 \end{itemize}
In the following, for simplicity, we often use the abbreviation 
\begin{equation}
    \label{A}
    A \coloneqq (J-R)Q.
\end{equation}

To compute the Lie bracket expressions for problem \eqref{linearpHOCP}, we again add an auxiliary state variable $x_{n+1}$ to transform the problem into  Mayer form (with only terminal cost)  given by the following dynamics
\begin{align*}
\dot x_{n+1} & = y^T Y y + y^T N u  + \ell^T u =  x^T Q G Y G^T Q x  + x^T Q G N u + \ell^T u,\\
x_{n+1} (0) &= 0.
\end{align*}
Then we set
\begin{equation*}
    f_0 := 
\begin{bmatrix}
Ax \\
x^T Q G Y G^T Q x
\end{bmatrix},\quad
f_i := 
\begin{bmatrix}
g_i \\
x^T Q G n_i + \ell_i
\end{bmatrix},
\end{equation*}
for $i=1,\dots,m,$ thus $f_0,f_1,\dots, f_m$
are functions from $\mathbb{R}^{n+1}$ to $\mathbb{R}^{n+1}$.
We obtain
\[
[f_i,f_j] = 
\begin{bmatrix}
0 \\
\vdots\\
0\\
n_j^T G^T Q g_i - n_i^T G^T Q g_j
\end{bmatrix},\quad i,j=1,\dots,m,
\]
so that
\begin{equation}
\label{Lielinear}
 \begin{bmatrix}
     \lambda \\
     1
 \end{bmatrix}^T [f_i,f_j] = n_j^T G^T Q g_i - n_i^T G^T Q n_j,  
\quad i,j=1,\dots,m,
\end{equation}
and we get the following Goh condition.
\begin{theorem}[Goh condition for linear pH systems]
\label{GohlinearpH}
For the optimal control problem \eqref{linearpHOCP}, 
\begin{equation}
\label{GohpH}
N^T G^T Q G  \text{ is symmetric}
\end{equation}
 along optimal trajectories.
\end{theorem}
\begin{proof}
The result follows from Theorem \ref{ThGoh}.
\end{proof}

An optimality condition for linear pH systems has been studied in \cite{SchPFWM23}. In order to compare our results with those in \cite{SchPFWM23}, we see that
the optimality condition \eqref{GohpH} is automatically satisfied in their framework, since there $N=I_m$ is a standing hypothesis.
\SA{Observe that, while in \cite{SchPFWM23} the terminal point is fixed—requiring the imposition of a normality condition—our problem \eqref{linearpHOCP} does not include terminal constraints. Although this may be seen as a limitation in certain contexts, it has the advantage of ensuring normality without additional assumptions.} The result also generalized the tracking control result discussed in \cite{SchZBSTW24}.

We next compute the expressions of the Lie brackets for the linear problem \eqref{linearpHOCP} in order to display the Goh and Legendre-Clebsch conditions. For $i=1,\dots,m$ we get
\[
[f_0,f_i] =
\begin{bmatrix}
-Ag_i \\
n_i^T G^T QA x - x^T Q G (M+M^T) G^T Q g_i
\end{bmatrix},
\]
%
\begin{multline*}
 \left[ f_0,[f_0,f_i] \right] = 
\begin{bmatrix}
A^2 g_i \\
n_i^T G^T Q A^2 x - 2g_i^T Q G M G^T Q A x + 2 x^T Q G M G^T Q A g_i
\end{bmatrix},
\end{multline*}
%
and
\begin{eqnarray*}
    d_i & =& \begin{bmatrix}
     \lambda \\
     1
 \end{bmatrix}^T \left[ f_0,[f_0,f_i] \right] \\
    &=& 
  \lambda^T A^2 g_i + n_i^T G^T Q A^2 x - 2 g_i^T Q G M G^T Q A x  +  2 x^T Q G M G^T Q A g_i .
\end{eqnarray*}
Additionally, for $ i,j=1,\dots,m$, we have 
\[
\left[ f_j,[f_0,f_i] \right] = 
\begin{bmatrix}
0 \\
n_i^T G^T Q A g_j - 2 g_i^T Q G MG^T Q b_j + n_j^T G^T QAg_i
\end{bmatrix}   
\]
and
\[
\begin{bmatrix}
     \lambda \\
     1
 \end{bmatrix}^T \left[ f_j,[f_0,f_i] \right] = 
 n_i^T G^T Q A G_j - 2g_i^T Q G M G^T Q g_j + n_j^T G^T QAg_i  .
 \]

Consequently, the matrix $W$ in \eqref{linearpHOCP} is given by
\[
W = N^T G^T Q A G - 2G^T Q G M G^T Q G + (N^T G^T QAG)^T.
\]
Inserting the structure of the matrix  $A$ (see \eqref{A}) and using the skew-symmetry of $J$, we get
\begin{equation}
   \label{W}
W = 
-N^T G^T Q R Q G - 2 G^T Q G M G^T Q G  
- (N^T G^T Q R Q G)^T,  
\end{equation}
%
and we obtain the following {\em Legendre-Clebsch condition.}
\begin{theorem}[Generalized Legendre-Clebsch condition for linear pH problems]
\label{PropgenLC}
Whenever the optimal control problem \eqref{linearpHOCP} of the linear port-Hamiltonian admits an optimal solution, the matrix $W$ given in \eqref{W}
is negative semidefinite.    
\end{theorem}

In the often considered special case of {\em minimizing the supplied energy,} that is, setting
\begin{equation}
\label{MellN}
M=0,\quad \ell=0,\quad N=I,
\end{equation}
one gets 
\begin{equation}
\label{Wsupply}
    W = -2 G^T QRQG,
\end{equation}
which is positive semidefinite, in view of the standing hypotheses for $Q$ and $R$. Moreover, whenever $R$ is definite, so  will be $W,$ leading then to a {\em feedback expression for the singular control from \eqref{Wd}.}
In the latter case, along such trajectories, the solution is asymptotically stable since the dissipation term acts in this direction. 
Additionally, under the hypotheses used in \cite{SchPFWM23}, namely, \eqref{MellN}  together with
\begin{equation}
\label{HypSchaller}
{\rm Im}(G) \cap {\rm Ker}(RQ) = 0,
\end{equation}
%
the expression of $W$ also coincides with  \eqref{Wsupply}, 
which is negative definite in view of assumption \eqref{HypSchaller}. Here we see directly that under the conditions \eqref{MellN} and
\eqref{HypSchaller} (as in \cite{SchPFWM23}), a feedback solution exists. On the other hand our analysis is more general, since we get  analogous results under noticeably weaker assumptions.

\begin{remark}\label{RemarkHigherOrder}{\rm 
    In Theorems \ref{GohlinearpH} and \ref{PropgenLC}, we have derived the Goh and generalized Legendre-Clebsch conditions for linear port-Hamiltonian (pH) systems, which correspond to second-order optimality conditions. \SA{The singular control is characterized when the singularity is of {\em order 1} (i.e., the control appears explicitly in the derivative $\frac{d^2}{dt^2} \mathcal{H}_{u}$).} Higher-order conditions (see, e.g., \cite{Knobloch80,Krener1977}) involve increasingly complex higher-order Lie brackets. Even for linear pH systems, these expressions become overly intricate, limiting their practical utility and insight.
    }
\end{remark}
\begin{remark}[Cost functional with reference trajectory]
\label{rem:reftra}
{\rm 
In many applications one wants to control a system towards a given reference trajectory  $x_{\rm ref}$. Replacing $x=\tilde x-x_{\rm ref}$ one gets a given inhomogeneity in the state and the output equation.
In this case, 
the system results in
 \begin{equation}
\begin{aligned}
\dot{\tilde{x}} &= (J - R) Q \tilde{x} + G u + h(t), \\
y &= G^T Q (\tilde{x} - x_{\rm ref}) = G^T Q \tilde{x} + k(t).
\end{aligned}
\label{linearpHref}
\end{equation}
 After removing the terms depending only on the reference trajectory, the cost functional takes the form
\begin{equation}
\label{costref}
\begin{aligned}
    \int_0^T \big[ & (x - 2x_{\rm ref})^T Q G Y G^T Q x  + (x - x_{\rm ref})^T Q G N u + \ell^T u \big] \, dt.
\end{aligned}
\end{equation}
For the optimal control problem governed by system \eqref{linearpHref} and associated to the cost \eqref{costref}, the expressions for the Lie brackets coincides with \eqref{Lielinear}, the matrix $W$ matches with \eqref{W}, so that the optimality conditions do not change.
Additionally, the vector $d$ is given by
\[
\begin{aligned}
d_i =\ & \lambda^T A^2 g_i + n_i^T G^T Q A^2 x - 2 g_i^T Q G Y G^T Q A x  + 2 (x - x_{\rm ref})^T Q G Y G^T Q A g_i.
\end{aligned}
\]
}
\end{remark}

\begin{example}\label{ex:MDK}

{\rm To illustrate the elegance of the optimality conditions, consider the following example of a second order linear control problem that is discussed in detail in \cite{FauKMPSW25} and arises e.g. in the control of high rise buildings~\cite{WarsBohmSawoTari21}:
\begin{align}\label{msd}
M\tfrac{\mathrm{d}^2}{\mathrm{d}t^2}{q} + D\tfrac{\mathrm{d}}{\mathrm{d}t}{q} + Kq = Bu
\end{align}
with real symmetric positive definite matrices $M,K$ and $B$ of full column rank. In contrast to \cite{FauKMPSW25} we assume that the damping does not work on the whole position vector, i.e., that $D$ is only positive semidefinite.  Setting $x=\begin{bmatrix} x_1 \\ x_2 \end{bmatrix}=\begin{bmatrix} M \dot q \\ q \end{bmatrix}$,
$
R=\begin{bmatrix}   D  & 0
\\ 0 & 0\end{bmatrix}$, $J=\begin{bmatrix}  0 & 
-I \\
I & 0\end{bmatrix}$, $Q=\begin{bmatrix} M^{-1} & 0 \\ 0 & K \end{bmatrix}$, $G=\begin{bmatrix} B  \\ 0 \end{bmatrix}$, 
and  introducing a collocated output $y=G^T Q x$, the system has the form \eqref{linearpH}.
Unfortunately, the condition ${\rm Im}(G) \cap {\rm Ker}(RQ) = \{0\}$ that is used in~\cite{SchPFWM23} does not hold, since $DM^{-1}$ has a kernel and $W=-2B^TM^{-1} D M^{-1}B$ is only semidefinite. Hence the techniques in ~\cite{SchPFWM23} cannot be applied. Also the linear system with $W$ for the feedback  is, in general, not (uniquely) solvable. However, still a (non-unique) feedback solution may exist  if the right hand side $d$ is in the image of $W$, see \cite{FauKMPSW25,Meh91}. 
In case $M$ is only semidefinite  then one can rewrite the system in descriptor form as in \eqref{daemechsys} and the presented procedures apply to this case. We then obtain  the optimality conditions as in Section~\ref{sec:dae}.}
\end{example}
In this subsection we have seen that, by using the general theory, we can generalize the conditions  for optimality presented in \cite{SchPFWM23}. In the next subsection we consider the case of control bounds.
\subsection{Optimality conditions for linear pH problems  
subject to control bounds}

If we require control bounds, i.e.
\[
\mathbb{U} = [\underline{u}_1,\overline{u}_1] \times \dots [\underline{u}_m,\overline{u}_m],
\]
so that $u$ is subject to the inequalities
\begin{equation}
\underline{u}_i \leq u_i(t) \leq \bar{u}_i,\quad i=1,\dots,m,
\end{equation}
then we obtain the optimal control problem
\begin{equation}
\label{linearpHconstraints}
    \begin{split}
         \min \,\, &\int_0^T y^T Y y + y^T N u + \ell^T u \\
         {\rm s.t.}\,\,\,\,   &\dot x(t) = (J-R) Q x(t) + Gu(t),\\
         & x(0) = x^0,\\
        & y(t) = G^T Q x(t) ,\\
        & \underline{u}_i \leq u_i(t) \leq \bar{u}_i,\quad i=1,\dots,m.
    \end{split}
\end{equation}
In order to study \eqref{linearpHconstraints}, we consider \emph{switching functions}. 
For the $i$-th component of the control, the switching function is given by
\begin{equation}\label{swfun}
s_i \coloneqq \mathcal{H}_{u_i} = \begin{bmatrix}
     \lambda \\
     1
 \end{bmatrix}^T f_i =  \lambda^T g_i +  x^T Q G n_i + \ell_i .
\end{equation}
Introduce the {\em set of switching points}
\[
Z_i \coloneqq \{ t \in [0,T] : s_i(t) =0\},
\]
and the {\em interior set} for control variable $i,$ for $i=1,\dots,m:$
\[
W_i \coloneqq \{ t \in [0,T] : \underline{u}_i < u^*_i(t) < \overline{u}_i\}.
\]
In view of the maximum condition of Pontryagin's Maximum Principle \cite{Pontryagin}, one then has that $W_i$ is contained, up to a set of measure zero, in $Z_i$.

Now, following the notation in  \cite[Theorem 8]{SchPFWM23}, for any subset $\mathcal{I} \subseteq \{1,\dots,m\}$, we set
\[
u_{\mathcal{I}} \coloneqq (u_i)^T_{i\in \mathcal{I}},\qquad
G_{\mathcal{I}} \coloneqq (g_i)^T_{i\in \mathcal{I}},\qquad
s_{\mathcal{I}} \coloneqq (s_i)^T_{i\in \mathcal{I}},
\]
and analogous notations are used for other involved matrices and vectors.

At this point, it is useful to present the adjoint equation for \eqref{linearpHconstraints} which reads
\begin{equation*}
    \left\{
    \begin{array}{rl}
        \displaystyle \frac{d}{dt} \lambda^T &= - \lambda^T A = - \lambda_{n+1}\left( x^T Q G (Y + Y^T) G^T Q + u^T N^T G^T Q \right), \\[1ex]
        \dot{\lambda}_{n+1} &= 0,
    \end{array}
    \right.
\end{equation*}
recalling that $\lambda_{n+1} \equiv 1,$ since the involved terminal cost is  $x_{n+1}(T).$

\begin{theorem}[Goh and generalized Legendre-Clebsch conditions for linear pH problems under control constraints]
\label{PropGohlinear}
For problem \eqref{linearpHconstraints}, the Goh conditions read 
\begin{equation}
\begin{bmatrix} \lambda \\ 1 \end{bmatrix}^T [f_i,f_j] =  n_j^T G^T Q g_i - n_i^T G^T Q n_j  =0, 
\end{equation} 
on any open subinterval $(a,b)$ of $[0,T]$ on which 
\[
\underline{u}_i < u^*_i(t) < \overline{u}_i \quad \text{ and } \quad \underline{u}_j < u^*_j(t) < \overline{u}_j.
\]
The generalized Legendre-Clebsch condition states that the matrix
\begin{equation}
   N_{\mathcal{I}}^T G^T QAB_{\mathcal{I}} - 2 G_{\mathcal{I}}^T QGY G^T Q G_{\mathcal{I}} + \left(N_{\mathcal{I}}^T G^T QAG_{\mathcal{I}}\right)^T
\end{equation}
is negative semidefinite on any open interval $(a,b) \subset [0,T]$ that is contained in  $\bigcap_{i\in \mathcal{I}} W_i$ for $\mathcal{I} \subseteq \{1,\dots,m\}$, i.e., $\mathcal{I}$ is any subset of indexes of simultaneous singular components of the control. 
\end{theorem}

\begin{proof}
The proof follows from Theorems \ref{PropGohnonlinearpHconstraints} and \ref{GLCnlconstraints}.
\end{proof}

\begin{example}[Special cases $N=0$ and $N=I_m$]
\label{rem:speccase}{\rm 
When $N=0$ or $N=I_m$, using the notation introduced in \eqref{deffi} for the vector fields $f_i,$ one has that $[f_i,f_j] =0$ in the whole space $\mathbb{R}^n,$ independently of the trajectories. Therefore, {\em the Goh conditions hold trivially.}
Additionally, for $N=0,$ the matrix in the generalized Legendre-Clebsch condition of Theorem \ref{PropGohlinear} reads
\begin{equation}
      - 2 G_{\mathcal{I}}^T QGY G^T Q G_{\mathcal{I}}, 
\end{equation}
and for $N=I_m,$ it has the form
\begin{equation}
     - 2 G_{\mathcal{I}}^T Q[R+GY G^T] Q G_{\mathcal{I}}. 
\end{equation}
}
\end{example}

We  obtain the following feedback formulas.
\begin{theorem}[Feedback formulas for singular controls]
    For any subset $\mathcal{I} \subseteq \{1,\dots,m\},$ given an open interval $(a,b) \subseteq \bigcap_{i\in \mathcal{I}} W_i ,$ the following expressions for the control hold along optimal trajectories of problem \eqref{linearpHconstraints}
\begin{small}
  \begin{equation*}
\begin{aligned}
&    0 =  \lambda^T A^2 G_{\mathcal{I}}  + x^T \left[2 QG Y G^T QAG_{\mathcal{I}} 
-2A^T QGY G^T Q G_{\mathcal{I}} 
+ A^T A^T  Q G N_{\mathcal{I}} \right] \\
& +  u^T_{\mathcal{I}} \left[ N_{\mathcal{I}}^T G^T QAG_{\mathcal{I}} - 2 G_{\mathcal{I}}^T QGY G^T Q G_{\mathcal{I}} + \left(N_{\mathcal{I}}^T G^T QAG_{\mathcal{I}}\right)^T \right]\\
& +  u^T_{\mathcal{A}} \left[ N_{\mathcal{A}}^T G^T QAG_{\mathcal{I}} - 2 G_{\mathcal{A}}^T Q GY G^T Q G_{\mathcal{I}} + G_{\mathcal{A}}^T A^T  Q G N_{\mathcal{I}} \right].
\end{aligned}
\end{equation*}
\end{small}
    Additionally, for the special case $N=0$, one gets the reduced expression
\begin{eqnarray*}
0 
&=&  \lambda^T A^2 G_{\mathcal{I}} + x^T \left[2 QG Y G^T QAG_{\mathcal{I}} 
-2A^T QGY G^T Q G_{\mathcal{I}} 
 \right] \\
&& \quad -2  (u^T_{\mathcal I}  G_{\mathcal I}^T + u^T_{\mathcal A}  G_{\mathcal A}^T )QGY G^T Q G_{\mathcal{I}},
\end{eqnarray*}
and for $N=I_m$, 
\begin{eqnarray*}
0
&=&  \lambda^T A^2 G_{\mathcal{I}}  \\
&+& x^T \left[2 QG Y G^T Q (J-R) 
-2A^T QGY G^T  
+ A^T A^T   \right] Q G_{\mathcal{I}} \\
& &\quad -2 (u^T_{\mathcal I}  G_{\mathcal I}^T + u^T_{\mathcal A}  G_{\mathcal A}^T ) Q \left[   R  + G Y G^T  \right] Q G_{\mathcal{I}}.
\end{eqnarray*}
\end{theorem}

\begin{proof}
Following the notation introduced above, and from \eqref{swfun}, one gets 
\[
s_{\mathcal{I}} =  \lambda^T G_{\mathcal{I}} + \lambda_{n+1} \left( x^T Q G N_{\mathcal{I}}  + \ell_{\mathcal{I}}  \right) .
\]
Therefore,
\begin{small}
\begin{eqnarray}
\label{dotsI}
        \dot{s}_{\mathcal{I}} &=&  \frac d{dt}\lambda^T G_{\mathcal{I}} + \lambda_{n+1}  \dot{x}^T Q G N_{\mathcal{I}}  \nonumber \\
        &=& - \lambda^T A G_{\mathcal{I}} - \lambda_{n+1} x^T Q G (Y+Y^T) G^T Q G_{\mathcal{I}} \\ && - u^T \underbrace{N^T G^T Q G_{\mathcal{I}}}_{(N^T G^T Q G)_{\mathcal{I}}} + x^T A^T Q G N_{\mathcal{I}} + \lambda_{n+1} u^T \underbrace{G^T Q G N_{\mathcal{I}}}_{(G^T QGN)_{\mathcal {I}}}.\nonumber 
\end{eqnarray}
\end{small}
We split the remainder  of the proof in two cases. 

{\bf Case 1:  $N=0$ or $N = I_m$.} From \eqref{dotsI} one gets that
\begin{equation}
    \dot{s}_{\mathcal{I}} =  - \lambda^T A G_{\mathcal{I}} -  x^T Q G (Y+Y^T) G^T Q G_{\mathcal{I}} + x^T A^T Q G N_{\mathcal{I}}. 
\end{equation}
Thus,
\begin{small}
\begin{equation}
\label{ddotsI}
\begin{aligned}
\ddot s_{\mathcal{I}} 
&=  \lambda^T A^2 G_{\mathcal{I}} + x^T \left[2 QG Y G^T QAG_{\mathcal{I}} 
-2A^T QGY G^T Q G_{\mathcal{I}} 
+ A^T A^T  Q G N_{\mathcal{I}} \right] \\
 +&  u^T \left[ N^T G^T QAG_{\mathcal{I}} - 2 G^T QGY G^T Q G_{\mathcal{I}} + G^T A^T  Q G N_{\mathcal{I}}\right].    
\end{aligned}
\end{equation}
\end{small}
For $N=0$, splitting the control vector in $u_{\mathcal{I}}$ and $u_{\mathcal{A}},$ one gets from latter equation that
\begin{small}
\begin{equation*}
\begin{aligned}
\ddot s_{\mathcal{I}} 
=&  \lambda^T A^2 G_{\mathcal{I}} + x^T \left[2 QG Y G^T QAG_{\mathcal{I}} 
-2A^T QGY G^T Q G_{\mathcal{I}} 
 \right] \\
&  -2  (u^T_{\mathcal I}  G_{\mathcal I}^T + u^T_{\mathcal A}  G_{\mathcal A}^T )QGY G^T Q G_{\mathcal{I}},
\end{aligned}
\end{equation*}
\end{small}
where $\mathcal A \coloneqq \mathcal{I}^c.$
For $N=I_m$, from \eqref{ddotsI}, one derives   
\begin{small}
\begin{equation*}
\begin{aligned}
\ddot s_{\mathcal{I}} 
=&  \lambda^T A^2 G_{\mathcal{I}} 
+ x^T \left[2 QG Y G^T Q (J-R) 
-2A^T QGY G^T  
+ A^T A^T   \right] Q G_{\mathcal{I}} \\
 -&2  (u^T_{\mathcal I}  G_{\mathcal I}^T + u^T_{\mathcal A}  G_{\mathcal A}^T ) Q \left[   R  + G Y G^T  \right] Q G_{\mathcal{I}},
\end{aligned}
\end{equation*}
\end{small}
where we replaced the matrix $A$ by its expression in \eqref{A}.

{\bf Case 2:  Arbitrary matrix $N$.} 
From expression \eqref{dotsI} for $\dot{s}_{\mathcal I}$,  one obtains
\begin{small}
\begin{equation*}
\begin{aligned}
    \ddot s_{\mathcal{I}} 
= & \lambda^T A^2 G_{\mathcal{I}} + x^T \left[2 QG Y G^T QAG_{\mathcal{I}} 
-2A^T QGY G^T Q G_{\mathcal{I}} 
+ A^T A^T  Q G N_{\mathcal{I}} \right] \\
& +  u^T_{\mathcal{I}} \left[ N_{\mathcal{I}}^T G^T QAG_{\mathcal{I}} - 2 G_{\mathcal{I}}^T QGY G^T Q G_{\mathcal{I}} + \left(N_{\mathcal{I}}^T G^T QAG_{\mathcal{I}}\right)^T \right]\\
& +d +  u^T_{\mathcal{A}} \left[ N_{\mathcal{A}}^T G^T QAG_{\mathcal{I}} - 2 G_{\mathcal{A}}^T Q GY G^T Q G_{\mathcal{I}} + G_{\mathcal{A}}^T A^T  Q G N_{\mathcal{I}} \right].
\end{aligned}
\end{equation*}
\end{small}
~
\end{proof}

\subsection{Linear descriptor systems}\label{sec:descriptorlinear}
The results in this section can be directly applied to the class of index-one descriptor systems as in the nonlinear case. One obtains an explicit formula for $x_2$ as
\begin{small}
\[ 
x_2=-[(J_{22}-R_{22})Q_{22}]^{-1}[(J_{22}-R_{22})Q_{21}+
(-J_{12}^T-R_{12}^T)Q_{11}]x_1.
\]
\end{small}
Inserting this into the first equation gives an ordinary  pH system  only in $x_1$ of the form 
\begin{equation*}
\begin{split}
     \dot x_1&= (\tilde J_{11}-\tilde R_{11})Q_{11} x_1 +\tilde G_1 u,\\
y&= \tilde G_1^T Q_{11} x_1,
\end{split}
\end{equation*} 
and $x_2$ can be determined in a post-processing step.

The cost functional is then
\[
\mathcal J(x_1,u)=\int_{t_0}^{T} x_1^T(\tau)  Q_{11}^T \tilde Gu(\tau) d\tau
\]
and thus we have again reduced the problem
to one for which the previous formulas apply. 

\section{Conclusions and future work}\label{sec:conclusion}

For the singular optimal control problem of controlling a port-Hamiltonian system with the supplied energy as cost function, the Goh optimality conditions are derived. It is shown that the use of the specific port-Hamiltonian structure leads to very elegant solution formulas, in particular in the case of linear port-Hamiltonian systems with quadratic Hamiltonian. Optimality conditions are also presented for the case of control constraints. 

It remains an open problem how to give necessary and sufficient conditions for the existence of a feedback control in the case when the matrix $W$ in the generalized Legende-Clebsch condition is only semidefinite. This is closely related to the topic  to derive the optimality conditions for general port-Hamiltonian descriptor systems.

Another classical application of optimal control is the computation of (feedback) controls $u=k(x)$ or output feedbacks $u=k(y)$ so that the closed loop system is asymptotically stable. To achieve this, one just chooses an infinite interval, i.e. $T=\infty$, and requires that for the closed loop solution, $\lim_{t\to \infty} x(t)=0$. For this case, analogous results to the ones obtained in this work are expected to hold.
It is also an interesting question what happens when the usual approach, of adding a quadratic penalty term   $\alpha \sum_{j=1}^m u_i^2$ to the cost function is used, and one considers the convergence behavior for $\alpha \to 0$, see \cite{JacGL70}. The Goh conditions will not appear in that context, since $\mathcal{H}_u$ depends on the control and, consequently, it is, in general, not differentiable in time.

We have studied problems without terminal constraints, which admit a unique multiplier that verifies \emph{normality}, i.e., the multiplier associated to the integral cost can be normalized to $1$. 
In  \cite{Aronna2018,ABDL12}  a framework for dealing with general terminal constraints of the form
$\varphi_i(x(T))\leq 0,$ $\eta_j(x(T))= 0,$ for $i=1,\dots, d_\varphi,$ and $ j=1,\dots, d_\eta$ was introduced. Also, \cite{FraTon13} consider terminal constraints of the form $x(T) \in K$ under an additional hypothesis. 

Using the techniques in \cite{Aronna2018,ABDL12} one may as well discuss the derivation of second order sufficient conditions. 

\section*{Acknowledgment}

M.S.A. acknowledges the financial support of the Brazilian agencies  FAPERJ processes E-26/203.223/2017, E-26/201.346/2021 and 210.037/2024,  CNPq  process 310452/2019-8 and CAPES process 88881.162133/2017-01. The current research was mainly developed while M.S.A. was holding a  Humboldt Fellowship  at TU Berlin (Germany).
V.M. has been supported by Deutsche Forschungsgemeinschaft (DFG) through the SPP1984 ``Hybrid and multimodal energy systems'' Project: \emph{Distributed Dynamic Security Control} and by the DFG Research Center Math+, Project  No.~390685689. 
\emph{Advanced Modeling, Simulation, and Optimization of Large Scale Multi-Energy Systems}.

%

%
%

\bibliographystyle{plain}
\bibliography{generalsoledad}

\end{document}